\documentclass[11pt]{article}
\usepackage{amsmath,amssymb,amsthm,color,comment,geometry,graphics,hyperref}
\usepackage[T1]{fontenc}
\usepackage[utf8]{inputenc}
\usepackage{authblk}

\newtheorem{theorem}{Theorem}
\newtheorem{lemma}{Lemma}

\theoremstyle{definition}
\newtheorem{definition}{Definition}
\theoremstyle{remark}
\newtheorem{remark}{Remark}

\title{Parseval's Identity and Values of Zeta Function at 
Even Integers}

\author[$\dag$]{Asghar Ghorbanpour \thanks{The major part of this project was completed while A. Ghorbanpour was a PIMS post-doctoral fellow at the University of Regina.}}
\affil[$\dag$]{Department of Mathematics, 
Western University}
\author[$\ddag$]{ Michelle Hatzel }
\affil[$\ddag$]{Department of Mathematics and    Statistics, 
University of Regina}

\newcommand{\Addresses}{{
  \bigskip
  \footnotesize

  Asghar Ghorbanpour, \textsc{The Department of Mathematics,
Western University,
London, ON, Canada, N6A 5B7.}\par\nopagebreak
  \textit{E-mail address}:\texttt{aghorba@uwo.ca}

  \medskip

  Michelle Hatzel, \textsc{Department of Mathematics and Statistics, 
University of Regina, 
Regina, SK, Canada, S4S 0A2.}\par\nopagebreak
  \textit{E-mail address}: \texttt{hatzel2m@uregina.ca}

}}

\date{\today }
\begin{document}
\maketitle

\abstract{ 
Historically known as the Basel problem, evaluating the Riemann zeta function at two has resulted in numerous proofs, many of which have been generalized to compute the function's values at even positive integers.  
We apply Parseval's identity to the Bernoulli polynomials to find such values.
}

\section{Introduction}
The search for the sum of all the reciprocal squares,
$$\sum_{n=1}^\infty \frac{1}{n^2}=1+\frac{1}{4}+\frac{1}{9}+\frac{1}{16}+\cdots$$
is considered to have begun with Pietro Mengoli (1626-1686), who posed the challenge in \cite{mengoli1650} in 1650. 
Eventually it became known as the Basel problem, largely due to the attention given it by University of Basel professor Jakob Bernoulli (1654-1705). Bernoulli is reported to have written of it, 
``If somebody should succeed in finding what till now withstood our efforts and communicate it to us we shall be much obliged to him.''\footnote{The origin of this statement is attributed to \textit{Tractatus de Seriebus Infinitis}, a collection Bernoulli made of his own work on infinite series that was published in 1689.}

Bernoulli's words convey the difficulty of the Basel problem, but his statement is even more interesting given that Bernoulli himself discovered the key to solving it.  Without knowing the full significance of them, Bernoulli had derived formulae which gave the numbers that would become known as the Bernoulli numbers.  These formulae were published in 1713, in his posthumous text, \textit{Ars Conjectandi}, but it would be Leonhard Euler (1707-1783) who would use these numbers to finally answer Mengoli's challenge.

Euler was made aware of the Basel problem by Johann Bernoulli (1667-1748), his mentor and Jakob's younger brother; in papers presented from 1731-36, Euler expounded an original method of approximation to achieve the exact value of the series and proved that 
\begin{equation*}
\sum_{n=1}^{\infty}\frac{1}{n^2}=\frac{\pi^2}{6}.
\end{equation*}   
In addition, Euler discovered a second technique, which uses infinite products and partial fraction decomposition for approximating values of infinite series (see \cite{Ayoub1974} for details of these publications). 
Euler eventually refined his methods to precisely determine the sums of reciprocal series raised to even powers.
Arising in the computations are the numbers that Bernoulli discovered.  
Denoted as $B_{2k}$, we see Bernoulli numbers in the general formula,
\begin{equation}\label{generalformula}
\sum_{n=1}^\infty \frac{1}{n^{2k}}=
\frac{(-1)^{k-1}\pi^{2k} 2^{2k-1} }{(2k)!}B_{2k}, \quad k\geq 1.
\end{equation}
Up to this day no one knows the exact values of series of reciprocals raised to odd powers.

The Basel problem and the study of infinite series underwent its next significant transformation due to Bernhard Riemann (1826-1866).  
In 1859 Riemann presented \cite{Riemann1859},
which showed a new way of evaluating infinite power series 
through the introduction of the complex variable, $s$, into the sum, 
 $\sum \frac{1}{n^s}$.  
As such, this sum gives a holomorphic function on the right half-plane, $\Re (s) >1$.
Riemann also showed the function  extends to a meromorphic function on the complex plane, with a simple pole at $s=1$.
This function, denoted by $\zeta$, is called the  {\it Riemann zeta function}, where  \eqref{generalformula} gives its values at even integers, $\zeta(2k)$.

Following the success of Euler and with the importance Riemann imparted on it, interest in the Riemann zeta function has continued; different approaches to the Basel problem have led to several elementary methods for finding values of $\zeta (2)$ and $\zeta(2k)$, where $k$ is a positive integer \cite{Estermann1947,Williams1953,Apostol1973,Kalman1993}. 
These approaches are the result of seeing the problem from different perspectives furnished by various branches of mathematics.

We will consider one of these methods from Fourier analysis. 
To evaluate $\sum \frac{1}{n^2}$, Parseval's identity applied to $f(x)=x$ is a common textbook technique 
(for examples, see \cite[p.198]{Rudin1976} and \cite[p.440]{Titchmarsh1939}).
To apply the same approach for even integer values $\zeta (2k)$, for all $k \geq 1$,  
one requires the appropriate function whose absolute value of $n^{th}$ Fourier coefficient is $\frac1{n^k}$.  
We found that a series of functions to do the job is the Bernoulli polynomials.

The history of the Basel problem is much richer than we've been able to present here. 
We encourage the reader to consider the resources for this paper's introduction, in particular \cite{Ayoub1974} and \cite{Knoebel2007}.

We have structured the paper in four sections.  
In section two, by following the original work of Bernoulli,  Bernoulli numbers and polynomials are introduced and some of their properties are studied.
The next section includes a  theory of Fourier coefficients and Parseval's identity.
Then a brief geometrical interpretation of this identity is discussed by means of an introductory approach to Hilbert spaces.
In the last section, using properties of Bernoulli polynomials, their Fourier coefficients are computed.
Then Parseval's identity is applied and the values of the zeta function at even integers are computed (Theorem \ref{main}). 
The last section is concluded by some remarks on our proof and related works in the literature.
All sections are written to be accessible to undergraduate math students and we have tried to keep with the historical order.

{\bf Acknowledgment}: This paper is based on a poster presented by undergraduate student M. Hatzel at the Calgary Applied and Industrial Mathematical Sciences Conference, May 21-22, 2017, at which several people gave us feedback. 
Also, we would  like  to thank Professors  Masoud Khalkhali, Ram Murty and J\'{a}n Min\'{a}c
for  the valuable comments and encouragement that we received from them.

\section{Bernoulli Numbers and Polynomials}

The starting point of the Bernoulli polynomials goes back to the sum of powers of integer numbers. 
By the 6th Century B.C.E. the Pythagoreans knew how to find the sum of the first natural numbers, 
\begin{equation} \label{sumNaturalnumbers} 
 \sum _{n=1}^{m-1} n=\frac{1}{2}m(m-1)=\frac{m^2}{2} - \frac{m}{2}.
 \end{equation}
Archimedes (c. 287-212 B.C.E.) discovered how to calculate the sum of squares: \cite{Knoebel2007}  
\begin{equation}\label{sumSquares}
\sum _{n=1}^{m-1} n^2=\frac{1}{6}m(m-1)(2m-1)=\frac{1}{3}m^3-\frac{1}{2}m^2+\frac{1}{6}m.
\end{equation}
Finding sums of other powers
began to reach its climax in the 17th Century, with mathematicians such as Pierre de Fermat and Blaise Pascal coming closer to the objective. 
Then, Jacob Bernoulli discovered the right way of looking at the problem.

Let us first fix a notation\footnote{Bernoulli in \cite{Bernoullitrans2006} looks for sums of first $m$ numbers rather than $m-1$. 
This will introduce some slight differences between what we will find and what is available in Bernoulli's notes.}
\begin{equation}\label{notationforSp}
S_p(m):=\displaystyle{\sum_{n=1}^{m-1} n^p}.
\end{equation} 
In the study of binomial coefficients, Bernoulli found the following identity\footnote{A very good exercise for the interested reader would be to attempt a combinatorial proof for this identity.}
$$\sum_{n=0}^{m-1} \binom{n}{p}=\binom{m}{p+1}.$$
Note that  when expanded, the summand of the left side, 
$\frac{1}{p!}n(n-1)\cdots (n-p+1),$
 will give a polynomial  of degree $p$ in $n$. 
The sums of each term of this polynomial gives some $S_k(m)$.
Using this identity and by induction, Bernoulli found values of $S_p(m)$ for $p=1,\cdots,10$ \cite{Bernoullitrans2006}.
Furthermore, by an attentive examination of the these formulae, he discovered the pattern for coefficients of $S_p$.
This pattern is the main focus of the following theorem.
\begin{theorem}
Let $S_p$ be the quantity defined by \eqref{notationforSp}. 
Then $S_p(m)$'s are polynomials of order $p+1$ in $m$ and there is a sequence of rational numbers  $\{B_j\}_{j=0}^\infty$  
such that 
\begin{equation}\label{Spintermsofcoef}
 S_{p}(t)=\frac{1}{p+1}\sum_{j=0}^{p} B_{j} \binom{p+1}{j} t^{p-j+1}, \qquad  p>1.
\end{equation}
These numbers satisfy the following recursive relation
\begin{equation}\label{recursiveformulaforbj}
B_j=-\frac{1}{j+1}\sum_{l=0}^{j-1} B_{l} \binom{j+1}{l},\quad B_0=1.
\end{equation}  
\end{theorem} 
\begin{proof}
Let's first find a recursive formula for $S_p$. 
To do so we will apply a simple trick which is the change of index of summation in the definition of $S_p$ from $n$ to $n-1$:
\begin{equation*}
S_{p+1}(m+1)-1
=\sum_{n=2}^{m} n^{p+1}
=  \sum_{n=1}^{m-1} (n+1)^{p+1}
= \sum_{n=1}^{m-1} \sum_{k=0}^{p+1}{p+1 \choose k}n^k
= \sum_{k=0}^{p+1}{p+1\choose k} S_k(m). 
\end{equation*}
The above equality can be used to write 
$$(p+1) S_p(m)=S_{p+1}(m+1)-1-S_{p+1}(m)-\sum_{k=0}^{p-1}{p+1\choose k} S_k(m).$$
Using the simple fact that $m^{p+1}=S_{p+1}(m+1)-S_{p+1}(m)$, we find the recursive formula
\begin{equation}\label{recursiveformulaforSp}
S_p(m)=\frac{1}{p+1}\left(m^{p+1}-1-\sum_{k=0}^{p-1}{p+1\choose k} S_k(m)\right),
\end{equation}
where the initial value is given by $S_0(m)=m-1$.  
A direct result of this recursive formula is that $S_p$ is a polynomial of order $p+1$ for every $p$, with rational coefficients.
From now on we will change the integer variable $m$ to the general real variable $t$.

To prove \eqref{Spintermsofcoef} we will use induction on $p$ and construct $B_j$ as induction proceeds. 
As for the base case, we set $B_0=1$ and $B_1=-\frac{1}{2}$, then it is easily seen that $S_1(t)$ is of the form given by \eqref{Spintermsofcoef}.

Then, as the induction hypothesis, assume that there are rational numbers $\{B_j\}_{j=0}^{p-1}$ such that for all $k<p$ we have 
\begin{equation}\label{inductionhypothesis}
S_{k}(t) =\frac{1}{k+1}\sum_{l=0}^{k} B_{l} \binom{k+1}{l} t^{k-l+1}
 , \qquad  p>1
\end{equation}
One can readily see that $S_k(1)=0$ and using the induction hypothesis \eqref{inductionhypothesis}, the constants $B_k$, for $k<p$, satisfy the equality,
\begin{equation}\label{recursiveofBkininduction}
B_k=-\frac{1}{k+1}\sum_{l=0}^{k-1} B_{l} \binom{k+1}{l} 
 ,  \quad 1\leq k<p.
\end{equation}

We will proceed by  replacing the formula given for $S_k$ in \eqref{inductionhypothesis} into the recursive formula \eqref{recursiveformulaforSp}.
Notice that while replacing we should use $S_0(t)=t-1=B_0(t)-1$.
$$S_p(t)
=\frac{1}{p+1}\left(t^{p+1}-\sum_{k=0}^{p-1}\binom{p+1}{k} \frac{1}{k+1}\sum_{l=0}^{k} B_{l} \binom{k+1}{l} t^{k-l+1}\right).$$
By setting a new variable $j=p-k+l$ we can find the coefficient of  $t^{p-j+1}$  for any $1\leq j \leq p$ given by
$$\frac{-1}{p+1}\sum_{l=0}^{j-1}\binom{p+1}{p-j+l}\binom{p-j+l+1}{l} \frac{B_l}{p-j+l+1} =\frac{\binom{p+1}{j}}{p+1}\left(\frac{-1}{j+1}\sum_{l=0}^{j-1} \binom{j+1}{l} B_l\right).$$
Now we can use \eqref{recursiveofBkininduction}  for any 
$j<p$; however, for $j=p$ we let $B_p$ be defined by
$$B_p=-\frac{1}{p+1}\sum_{l=0}^{p-1} \binom{p+1}{l} B_l.$$
Therefore we have 
$$S_p(t)=\frac{1}{p+1}\left(t^{p+1}+\sum_{j=1}^p \binom{p+1}{j} B_j t^{p-j+1}\right).$$   
Noting that $B_0=1$, the induction step is complete. 
Observe that while proving the induction step, we constructed the sequence $B_p$ inductively such that the relation \eqref{recursiveformulaforbj} is satisfied.
\end{proof}

\begin{definition}
The constant $B_j$, obtained in the above theorem, is called the {\it $j^{th}$  Bernoulli number}.\footnote{Bernoulli originally denoted $B_2$ by $A$, and $B_3$ by $B$, so on and so forth.}
\end{definition} 
In the early 1730s, while 
proving his summation formula, 
Euler also discovered these numbers \cite{Euler1925}.
 Among the many of his discoveries was a recursive formula for finding the Bernoulli numbers, and a generating function.
 Here we shall use the recursive formula \eqref{recursiveformulaforbj} to show how the generating function can be computed. 
 
 Let $G(x)$ be the generating function of the Bernoulli numbers, i.e. formally $G(x)=\sum_{j=0}^\infty \frac{B_j}{j!}x^j$.
 Then we have
 \begin{eqnarray*}
 G(x)&=& \sum_{j=0}^\infty \frac{B_j}{j!}x^j\\
 &=& 1-\sum_{j=1}^{\infty} \frac1{(j+1)!}\sum_{l=0}^{j-1} B_{l} \binom{j+1}{l}x^j\\
 &=& 1-\sum_{j=1}^{\infty}\sum_{l=0}^{j-1} \frac{B_{l}}{(j+1-l)! l!} x^j\\
 &=& 1-\sum_{l=0}^{\infty}\frac{B_l}{l!}\sum_{j=l+1}^{\infty} \frac{1}{(j+1-l)!} x^{j}\\
 &=& 1-\sum_{l=0}^{\infty}\frac{B_l}{l!}\sum_{j=2}^{\infty} \frac{1}{j!} x^{j+l-1}\\
  &=& 1-\frac{1}{x}\left(\sum_{l=0}^{\infty}\frac{B_l}{l!}x^l\right)\left(\sum_{j=2}^{\infty} \frac{1}{j!} x^j\right)\\
  &=& 1-\frac{1}{x}G(x)\left(e^x-x-1\right).
\end{eqnarray*} 
  Therefore $G(x)= 1-\frac{1}{x}\left(e^x-x-1\right)G(x)$, which implies that
\begin{equation}\label{generatingfunctionnumbers}
G(x)=\frac{x}{e^x-1}. 
\end{equation}
From \eqref{generatingfunctionnumbers} one can find
$$B_0=1,\,\, B_1=-\frac12, \,\, B_2=\frac16, \,\, B_3=0, \,\, B_4=-\frac1{30}, \,\, B_5=0, \,\, B_6=\frac1{42}, \,\, B_7=0, \cdots. $$
Note  that $G(x)-(-\frac12 x)=\frac{x(e^x+1)}{2(e^x-1)}$ is an even function. 
This implies that all the odd Bernoulli numbers, $B_{2k+1}$ for $k\geq 1$, are zero ($B_1$ is the exception).
  
\begin{definition}
The derivative of the polynomial $S_p(t)$ is called the {\it $p^{th}$ Bernoulli polynomial} and we denote it by $B_p(t)$.
\end{definition}
Bernoulli  polynomials are monic polynomials and they can be written in terms of Bernoulli numbers  as follows (derive \eqref{Spintermsofcoef}):
\begin{equation}\label{Bernoullipolyintermsofnumbers}
B_p(t):=\sum_{k=0}^{p} B_k\binom{p}{k}t^{p-k}, \quad k\geq 0. 
\end{equation}
Using \eqref{Bernoullipolyintermsofnumbers} and \eqref{generatingfunctionnumbers}, one can easily find the generating function of the Bernoulli polynomials 
\begin{equation}\label{GeneratingfunctionofBpoly}
G(x,t)=\sum_{p=0}^\infty B_p(t)\frac{x^p}{p!}
= \frac{xe^{tx}}{e^{x}-1}.
\end{equation} 
Examples of the first few Bernoulli polynomials are
\begin{equation*}\
B_{0}(t)=1, \quad  B_{1}(t)=t-\frac{1}{2},\quad  B_{2}(t)=t^2-t+\frac{1}{6},\quad   B_{3}(t)=t^3-\frac{3}{2}t^2+\frac{1}{2}t.
\end{equation*}
By differentiating \eqref{Bernoullipolyintermsofnumbers} we have 
\begin{equation}\label{Berandderivatives}
B'_{p}(t)= pB_{p-1}(t),\quad p \geq 1.
\end{equation}
As a result, we have $S'_p(t)=\frac{B_{p+1}'(t)}{p+1}$, which can be used to write the sums of powers in terms of Bernoulli polynomials
\begin{equation}\label{sumsandBernoulli}
S_p(m)=S_p(m)-S_p(0)=\int_0^m S'_p(t)dt=\int_0^m\frac{B_{p+1}'(t)}{p+1}=\frac{1}{p+1}(B_{p+1}(m)-B_{p+1}(0)).
\end{equation}
Additionally, \eqref{Bernoullipolyintermsofnumbers} readily shows $B_p(0)=B_p$.
 Moreover, by \eqref{sumsandBernoulli} we have 
$$0=S_p(1)=\frac{1}{p+1}(B_{p+1}(m)-B_{p+1}(0)).$$
 Therefore,  
\begin{equation}\label{Bernumbers}
B_p(1)= B_p(0)=B_p, \quad p \geq 2.
\end{equation}
The reader can refer to  \cite{Apostol2008} for more details and more identities involving Bernoulli numbers and polynomials.

\section{Fourier Series and Parseval's Identity}

In this section, we will introduce the tools of Fourier coefficients and Parseval's identity which play a central role in our strategy to find values of the zeta function at even integers. 
Fourier analysis, at its original form, is concerned with the decomposition of functions as infinite sums of trigonometric functions.  
This branch of mathematical inquiry arose following Joseph Fourier (1768-1830) who, motivated by a need for formulae that could model the conduction of heat, used this technique to find real-value solutions of functions; they  also are used to measure frequencies of vibrations. 
Parseval's identity is named for Marc-Antoine Parseval (1755-1836), and deals with summability of the Fourier coefficients. 
From a different perspective, both of these tools are among the first versions of more abstract theory, that is 
the theory of Hilbert spaces.
We have chosen the latter to present the topic here; however, to avoid the technical complications we shall not include proofs and instead show similar results in the finite case to help the reader develop the right intuition for the topic.

\begin{definition}\label{fcoefficients}
Let $f$ be an integrable function on $[0,1]$ then the $n^{th}$ Fourier coefficient $c_n(f)$ of $f$ is defined by
\begin{equation*}
c_n(f):=\langle f,e_n\rangle=\int_0^1 f(t) e^{-2\pi int} dt,\quad n\in\mathbb{Z}.
\end{equation*} 
\end{definition}

To understand a 
 geometric meaning of the Fourier coefficients we need to see them in the general setting of Hilbert spaces,  
 which are complex vector spaces equipped with a Hermitian inner product with complete topology.  Further introduction to Fourier analysis in Hilbert space can be found in \cite{Saxe2002}. 
Let's first see the finite dimensional versions of such spaces $V=\mathbb{C}^k$ with the inner product given by 
$$\langle (v_1,\cdots,v_k),(w_1,\cdots,w_k)\rangle= \sum_{j=1}^k v_j\overline{w_j}.$$
Here $\overline{w_j}$ denotes the complex conjugate of the complex number $w_j$.
On such vector space we can define the length of vectors by 
\begin{equation}\label{normfrominnerproduct}
\|v\|:=\sqrt{\langle v,v\rangle},\qquad v\in V.
\end{equation}

Let $e_n$, for all $1\leq n\leq k$, denote the vector with 1 in the $n^{th}$ component and zero in all other components.
The set of vectors $\{e_n\}_{n=1}^m$ form a so-called orthonormal basis for $V$; that is, 
\begin{enumerate}
\item they are orthonormal: $\langle e_n,e_m\rangle=\delta_{nm}$;
\item every vector in $V$ can be written as a linear combination of $e_n$'s. 
\end{enumerate}
The important property of an orthonormal basis, in particular $\{e_n\}$, is that the coefficients of $e_n$ in the linear combination which gives the vector $v\in V$ is given by the inner product. 
In other words, if $v=\sum_{m=1}^k c_m e_m$ then 
$$\langle v,e_n\rangle= \langle \sum_{m=1}^k c_m e_m,e_n\rangle=\sum_{m=1}^k c_m \langle  e_m,e_n\rangle  =\sum_{m=1}^k c_m \delta_{nm}=c_n.$$ 
Moreover, the length of a vector can also be computed using its inner product by $e_n$ as follows
\begin{eqnarray}\label{finitepytho}\nonumber
\|v\|^2&=&\|\sum_{m=1}^k \langle v,e_m\rangle e_m\|^2\\ \nonumber
&=& \langle\sum_{n=1}^k \langle v,e_n\rangle e_n,\sum_{m=1}^k \langle v,e_m\rangle e_m\rangle\\
&=&\sum_{n=1}^k \sum_{m=1}^k  \langle v,e_n\rangle \overline{\langle v,e_m\rangle} \langle e_n, e_m\rangle\\\nonumber
&=& \sum_{n=1}^k   \langle v,e_n\rangle \overline{\langle v,e_n\rangle}\\\nonumber
&=& \sum_{n=1}^k   |\langle v,e_n\rangle|^2.
\end{eqnarray}
This equality is nothing but the Pythagorean theorem in higher dimensional Hermitian spaces.

To come back to our case, where the Fourier coefficients can be obtained, we will need infinite dimensional spaces with inner product, called Hilbert spaces. 
Consider the linear space 
$$H=\left\{\left.f:[0,1]\to\mathbb{C}\right| \, f \text{ is measurable and } \int_0^1|f(t)|^2dt<\infty\right\}.$$
Here the functions are going to play the role of vectors and the inner product is given by
$$\langle f,g\rangle=\int_0^1 f(t)\overline{g(t)}dt. $$
The norm of a function (called $L^2$-norm) is defined as \eqref{normfrominnerproduct} using this inner product.
Unlike $V$, $H$ is infinite dimensional, meaning that we need infinitely many elements $\{e_n\}_{j=1}^\infty$ to form a basis.
Also, we may encounter infinite sums while trying to write linear combinations of elements, 
so a notion of convergence of linear combinations will be needed.
In particular, the second criteria in the definition of orthonormal basis should be replaced by
\begin{enumerate}
\item[2'.] Every vector in $H$ is the limit of (possibly infinite) linear combinations of $e_j$'s. 
\end{enumerate}
As an example, the functions 
 \begin{equation*}
 e_n(t):=e^{2\pi i n t},\quad  n\in\mathbb{Z},
\end{equation*} 
form an orthonormal basis for $H$
(see examples in \cite[p.187]{Rudin1976}).

With all these in hand it is obvious that 
$$c_n(f)=\langle f,e_n\rangle.$$
Moreover, an infinite dimensional version of the computation \eqref{finitepytho}  can be performed and the result is known as Parseval's identity (for a proof see \cite[p.191]{Rudin1976}).   

\begin{theorem}[Parseval's identity]\label{TheorPars}
Suppose $f$ is a Riemann-integrable function. Then 
\begin{equation}\label{eqPars}
\int_0^1 |f(x)|^2dx=\sum_{-\infty}^\infty |c_n(f)|^2.\vspace*{-0.5cm}
\end{equation}\hfill\qedsymbol
\end{theorem}

Similar to \eqref{finitepytho}, Parseval's identity can be considered as the generalization of the Pythagorean theorem in infinite dimensional space $H$, where the absolute value of the Fourier coefficients $|c_n(f)|$ play the role of length of the orthogonal sides of (an infinite dimensional) right triangle, and the sum of their squares is equal to the square of length of the function (hypotenuse) given by the integral.

\begin{remark}
At the beginning of the 20th Century David Hilbert (1862-1943) introduced abstract inner product spaces to embrace existing theories of function spaces, such as Fourier analysis, and to develop new tools to study such notions as integral operators. 
In particular, these abstract spaces, known as Hilbert spaces, allow for the manipulation of functions which otherwise would not meet the conditions of convergence and continuity required to perform such manipulations.
\end{remark}

\section{The Main Theorem}

In this section we show how the values of $\zeta(2k)$ can be obtained by applying Parseval's identity with the the Fourier coefficients of Bernoulli polynomials.
To that end, we first need to find the Fourier coefficients of Bernoulli polynomials and their $L^2$ norm. 
In both of the following Lemmas, the main idea lies in the following simple computation for any differential function $f$ on $[0,1]$, which one can easily obtain by integration by parts and \eqref{Berandderivatives} and \eqref{Bernumbers}.
\begin{equation*}
\int_0^1B_k(t)f'(t)dt= (f(1)-f(0)) B_k-k\int_0^1 B_{k-1}(t)f(t)dt, \quad k\geq 2,
\end{equation*} 
while for $k=1$ we have $\int_0^1 B_1(t)f'(t)dt= (f(1)+f(0)) B_1-\int_0^1 f(t)dt.$

The following Lemma gives the Fourier coefficients of Bernoulli polynomials. 
\begin{lemma}\label{lemcnbk}
For all integers $k \geq1$,
$$c_n(B_k)=\begin{cases} \frac{-k!}{(2\pi i n)^k} & n\neq 0\\ 0 & n=0\end{cases}.$$
\end{lemma}
\begin{proof}
By integration by parts, for $n\neq 0$
{\small 
\begin{align*}
c_n(B_k)
&=\left. B_{k}(t)\frac{e^{-2\pi i n t}}{-2\pi i n}\right]_0^{1} -\int_{0}^{1} B'_k(t)\frac{e^{-2\pi i n t}}{-2\pi i n} \, dt.
 \end{align*}}
Because of \eqref{Bernumbers}, the first term vanishes for $k\geq 2$, and applying \eqref{Berandderivatives} the equality continues
\begin{equation}\label{recursive3}
c_n(B_k)=\frac{k}{2\pi i n}\int_{0}^{1} B_{k-1}(t)e^{-2\pi i n t}\, dt=\frac{k}{2\pi i n} c_n(B_{k-1}).
\end{equation}
And for $k=1$,  
\begin{equation}\label{recursive4}
c_n(B_1)=\frac{-1}{2\pi i n}.
\end{equation}
Combining \eqref{recursive3} and \eqref{recursive4} we achieve our proof, 
$$c_n(B_k)= \frac{-k!}{(2\pi i n)^k}.$$
Note that by the definition of the Bernoulli polynomials in terms of $S_p$ we have $\int_0^1B_k(t)dt=S_p(1)-S_p(0)=0$, so that $c_0(B_k)=0$, $k\geq 1$.
\end{proof}
\begin{remark}
Another interesting approach one can take to find the Fourier coefficients of Bernoulli polynomials  is to use their generating function \eqref{GeneratingfunctionofBpoly}. 
Being careful with the convergence conditions, one needs to see that 
$$\int_0^1 G(x,t) e^{-2\pi i nt}dt=\sum_{p=0}^\infty c_n(B_p) \frac{x^p}{p!}.$$ 
\end{remark}
To obtain $L^2$ norm of Bernoulli polynomials, we  first shall find a recursive expression for the integration of products of two Bernoulli polynomials.
\begin{lemma}\label{lembksq} 
For  all integers $1\leq k\leq l$, 
$$\displaystyle{\int_{0}^{1}} B_{k}(t)B_{l}(t) dt=\frac{(-1)^{k-1}l! k!B_{l+k}}{(l+k)!}.$$
\end{lemma}
\begin{proof}
Let's denote the left side by $A_{k,l}$.
Using \eqref{Berandderivatives} and integrating by parts we have,
\begin{equation*}
A_{k,l}
=\int_{0}^{1}\! B_{k}(t)\frac{B'_{l+1}(t)}{l+1} \, dt=\left. B_k(t)\frac{B_{l+1}(t)}{l+1}\right]_0^1 -\int_0^1 B'_k(t) \frac{B_{l+1}(t)}{l+1} dt.
 \end{equation*}
Because of \eqref{Bernumbers},  the first term vanishes for $k\geq 2$ and applying \eqref{Berandderivatives} the equality continues  
\begin{equation}\label{recursive1}
A_{k,l}=\frac{-k}{(l+1)}\int_0^1  B_{k-1}(t)B_{l+1}(t)dt=\frac{-k}{(l+1)}A_{k-1,l+1}.
\end{equation}
However, for $k=1$, we have  
\begin{equation}\label{recursive2}
A_{1,l}=\frac{B_{l+1}(0)}{l+1}\left( B_1(1)-B_1(0)\right) -\int_0^1 \frac{B_{l+1}(t)}{l+1} dt=\frac{B_{l+1}}{l+1}.
\end{equation}
Combining \eqref{recursive1} and \eqref{recursive2},
 we obtain the desired result: 
\begin{equation*}
A_{k,l}=\frac{(-1)^{k-1}l! k!B_{l+k}}{(l+k)!},\quad 0<k\leq l.
\end{equation*}
\end{proof}

Finally, we can have the main theorem.
\begin{theorem}\label{main}
For any positive integer $k \geq 1$,  we have 
\begin{equation}\label{valuesofzetaat2k}
\zeta(2k)=
\frac{(-1)^{k-1}\pi^{2k} 2^{2k-1} }{(2k)!}B_{2k}.
\end{equation}
\end{theorem}
\begin{proof} 
Applying Parseval's identity to $B_k$, $k \geq 1$, we have 
\begin{equation*}
\int_{0}^{1}| B_{k}(t)|^2 dt=\sum_{-\infty}^{\infty} |c_n(B_k)|^2.
\end{equation*}
The value of the left side, given by Lemma \ref{lembksq}, is equal to  
\begin{equation}\label{theorleft}
\int_{0}^{1} B_{k}(t)B_{k}(t) dt=\frac{(-1)^{k-1} (k!)^2B_{2k}}{(2k)!}.
\end{equation}
The sum on the right side, provided by Lemma \ref{lemcnbk}, gives us 
\begin{equation}
\sum_{-\infty}^{\infty} |c_n(B_k)|^2=\sum_{n\neq 0}\left| \frac{-k!}{(2\pi i n)^k}\right| ^2= 2 \frac{(k!)^2}{(2\pi)^{2k}}\sum_{n=1}^\infty \frac 1{n^{2k}}. \label{theorright}
\end{equation}
Putting together \eqref{theorleft} and \eqref{theorright} we  prove \eqref{valuesofzetaat2k}.
\end{proof}

We would like to finish this section with a few remarks on our proof and other related works.

\begin{remark}
Our work is not the first one which extends a method to evaluate the zeta function at two, to a general method to find $\zeta(2k)$, 
and in it to bring in Bernoulli polynomials;
 for example, see \cite{Ciaurri-Navas-Ruiz2015} where a telescoping series technique to find $\zeta(2)$, offered in \cite{Benko2012},  is generalized to find $\zeta(2k)$ using Bernoulli polynomials. 
\end{remark} 

\begin{remark}
Despite the very central role of Bernoulli polynomials in our work, there is nothing that made them unique in this process. 
In fact, there are infinitely many families of functions $f_k$ that can do the job. 
In fact, every function $f_k$ whose  Fourier coefficients are different than that of Bernoulli polynomials by a phase factor, $c_n(f_k)=e^{i\theta_{n,k}}c_n(B_k)$ with $\theta_{n,k}\in [0,2\pi]$, can be used here. 
On the other hand, a closer look at our proof reveals that property \eqref{Berandderivatives} is critical to it. 
For example, $f_k(x)=x^k$ is another  family of functions with the same property.\footnote{While preparing this paper, we became aware of the recently-posted paper \cite{Alladi-Defant2017} where Parseval's identity is applied on $x^k$ to find $\zeta(2k)$.}
\end{remark} 
\begin{remark}
Another technique in evaluating the zeta function  at even integers involves the pointwise convergence of Fourier series $\sum_{n\in \mathbb{Z}} c_n(f)e^{2\pi i n t}$ to $f(t)$. 
For a beautiful instance of this technique see \cite{Murty-Weatherby2016}.
\end{remark}

\begin{remark}
The functional equation of  the Riemann zeta function, 
\begin{equation}
\zeta (s)=2^{s}\pi ^{s-1}\ \sin \left({\frac {\pi s}{2}}\right)\ \Gamma (1-s)\ \zeta (1-s),
\end{equation}
where $\Gamma$ is the gamma function, plays a very important role in the study of the Riemann zeta function.
The functional equation relates the value of the zeta function  at $s$ to its value  at $1-s$.
Hence  we can see now that 
\begin{equation*}
\zeta(-2k+1)=-\frac{B_{2k}}{2k}, \quad k\geq 1
\end{equation*}
In fact, this is true for all negative integers and the odd Bernoulli numbers being zero gives the trivial zeros of the Riemann zeta function at negative even integers.
An interesting study, investigating the relation between the values of zeta at negative integers and functions $S_p$ can be found in \cite{Minac1994}.
\end{remark}


\providecommand{\bysame}{\leavevmode\hbox to3em{\hrulefill}\thinspace}
\providecommand{\MR}{\relax\ifhmode\unskip\space\fi MR }
\providecommand{\MRhref}[2]{%
  \href{http://www.ams.org/mathscinet-getitem?mr=#1}{#2}
}
\providecommand{\href}[2]{#2}

\Addresses

\end{document}